\documentclass[11pt]{article}

\usepackage{geometry}
\usepackage{amsmath,amsthm,amssymb,latexsym,amstext,amsfonts}
\usepackage{pst-node}
\usepackage{caption}
\usepackage{mathdots}
\usepackage{braket}
\usepackage{graphicx,pgf}
\usepackage{tikz}
\usetikzlibrary{snakes}
\usetikzlibrary{arrows}
\usepackage{epic}

\newtheorem{proposition}{Proposition}
\newtheorem{theorem}[proposition]{Theorem}

\newtheorem{lemma}[proposition]{Lemma}
\newtheorem{remark}[proposition]{Remark}
\newtheorem{definition}[proposition]{Definition}
\newtheorem{example}[proposition]{Example}

\def \RR{{\mathcal R}}

\def \LL{{\mathcal L}}
\def \DD{{\mathcal D}}

\def \EE{{\mathcal E}}

\def \CC{{\mathcal C}}

\newcommand{\Drel}{\mathbin{\mathcal D}}

\title{\emph{ON THE COSET CATEGORY \\ OF A SKEW LATTICE}}

\author{
Jo\~ao Pita Costa \\ 
Institut Jo\v zef Stefan,\\
Jamova Cesta 39, 1000 Ljubljana, Slovenia\\
joao.pitacosta@ijs.si
}
\date{\today}

\begin{document}

\maketitle

\begin{abstract}
Skew lattices are noncommutative generalizations of lattices. 
The coset structure decomposition is an original approach to the study of these algebras describing the relation between its rectangular classes. 
In this paper we will look at the category determined by these rectangular algebras and the morphisms between them, showing that not all skew lattices can determine such a category.
Furthermore, we will present a class of examples of skew lattices in rings that are not strictly categorical, and present sufficient conditions for skew lattices of matrices in rings to constitute $\wedge$-distributive skew lattices.
\bigskip

\medskip


\end{abstract}

\newpage

 

\section*{Introduction}

Skew lattices are one of the most successful generalizations of lattices, being non commutative but maintaining associativity, idempotency and four of the several possible absorption laws. The chosen absorption laws permit us to generalize several lattice theoretic concepts as is the case of distributivity, studied in \cite{Le89} or \cite{Le13}. The order structure of skew lattices finds a close relation to the order structure of its corresponding lattice.
These algebras can also be seen as double bands due to the fact that their reducts $(S;\wedge )$ and $(S;\vee)$ are regular semigroups of idempotents.
Green's relations take an important role in this research. In particular, $\DD$ is a congruence determining a decomposition deriving from Clifford-MacLean's result, that permits us to look at a skew lattice as a lattice of maximal rectangular algebras.  
  
The study of the coset structure of a skew lattice explores the interplay between related $\DD$-classes. It is an approach that has no counterpart in Semigroup Theory or in Lattice Theory. In the lattice case, the $\DD$-classes reduce to singletons, while in the case of bands those classes have a known impact in the study of the corresponding semi-lattice. This makes the approach a relevant method to study skew lattices, specifically.
Under certain conditions, such an algebra permits the construction of a category that has the $\DD$-classes as objects and the coset bijections between them as morphisms.
This paper explores questions posed in \cite{AAA80} regarding such a category, named coset category. In this paper we show that all cosets are rectangular subalgebras and that all coset bijections between them are in fact isomorphisms. Those isomorphisms describe the order structure of a skew lattice, as seen in \cite{AAA80}.  
This study of the characterization of several subvarieties of skew lattices by identities involving cosets, named \textit{coset laws}, started in \cite{Le93} and was continued in  \cite{Ka05} and in \cite{Co09a} . 

When considering a ring $R$, the operations defined by $x\wedge y=xy$ and $x\vee y=x+y-xy$ succeeded in providing a rather large class of examples of skew lattices which have motivated many of the properties studied in the general case. When $E(\mathbf{R})$ is  the set of all idempotent elements in a ring $\mathbf{R}$ and $S\subseteq E(\mathbf{R})$ is closed under both $\vee $ and $\wedge $, $(S;\wedge , \vee )$ is a skew lattice. Skew lattices in rings, together with skew Boolean algebras, constitute the largest classes of studied examples of skew lattices to date. 
Much has been done also in the particular case of skew lattices in rings of matrices by Cvetko-Vah and Leech in \cite{Ka05}, \cite{Ka05c} and \cite{Ka07} following the work of Radjavi and other authors on bands of matrices (see \cite{Fi94} and \cite{Fi99}).
It was shown in \cite{Ka05} that skew lattices in rings are not normal.
In the last section of this paper we will look at the coset category of skew lattices of matrices in rings and show that, in general, skew lattices in rings are also not strictly categorical. 
We will also present sufficient conditions for skew lattices of matrices in rings to constitute strictly categorical skew lattices and $\wedge$-distributive skew lattices.

This paper will be using: basic knowledge of Lattice Theory that can be retrieved in \cite{Ba67} dealing with lattice notions in a noncommutative context; several issues and results deriving from Semigroup Theory, that can be found in \cite{Ho76}, having in mind that we are dealing with bands of semigroups; and some Category Theory language, that can be revisited in \cite{La98}.

\section{Preliminaries}

A skew lattice is a set $S$ with binary operations $\wedge $ and $\vee $ that are both idempotent and associative, satisfying the absorption laws $x\wedge (x\vee y)=x=(y\vee x)\wedge x$ and their duals. 
 A \emph{band} is a semigroup of idempotents. 
Recall that a band is \emph{regular} if it satisfies $xyxzx=xyzx$, is \emph{normal} if it satisfies $xyzw=xzyw$, and is \emph{rectangular} if it satisfies $xyx=x$.
Any skew lattice $\mathbf S$ can be seen as double regular band by considering the band reducts $(S,\wedge )$ and $(S,\vee )$.
If these bands are rectangular we say that the skew lattice $\mathbf S$ is \emph{rectangular}. 
On the other hand, \emph{normal} skew lattices are the ones for which $(S;\wedge)$ is a normal band, and \emph{conormal} skew lattices are the ones for which $(S;\vee)$ is a normal band.  
A skew lattice is \emph{symmetric} whenever $x\wedge y= y\wedge x$ if and only if $x\vee y= y\vee x$.

Green's relations are five equivalence relations, introduced in  \cite{Gr51}, characterizing the elements of a semigroup in terms of the principal ideals they generate. 
Due to the absorption dualities, the Green's relations in the context of skew lattices are defined in  \cite{Le89}  by $\RR=\RR_{\wedge}=\LL_{\vee}$, $\LL=\LL _{\wedge}=\RR_{\vee}$ and $\DD =\DD _{\wedge}=\DD_{\vee}$. 
Right-handed skew lattices are the skew lattices for which $\RR =\DD$ while left-handed skew lattices are determined by $\LL =\DD$.

Two distinct concepts of order can be considered in a skew lattice $\mathbf{S}$: 
the \emph{natural partial order} defined by $x\geq y$ if $x\wedge y=y=y\wedge x$ or, dually, $x\vee y = x = y\vee x$; 
the \emph{natural preorder} defined by $x\succeq y$ if $y\wedge x\wedge y = y$ or, dually, $x\vee y\vee x = x$. 
Observe that $x\Drel y$ iff $x\succeq y$ and $y\succeq x$. 
Usually $\DD $ is referred in the available literature as the \emph{natural equivalence}. 

The fact that $\DD$ can be expressed by the natural preorder $\preceq$ allows us to draw diagrams, based on the Caley tables of the corresponding operations, that are capable of representing skew lattices as the one in Figure \ref{nonstcat}. 
An \emph{admissible Hasse diagram} of (a subset of) a skew lattice is a Hasse diagram for the natural partial order (usually represented by full edges) together with an indication of all $\DD $-congruent elements (usually represented by dashed edges). 
Unlike lattices, one such diagram can represent two distinct skew lattices (cf. \cite{Le89}).

Whenever $\mathbf{S}$ is a skew lattice, $\DD $ is a congruence, $\mathbf{S}/\DD $ is the maximal lattice image of $\mathbf{S}$ and all congruence classes of $\DD $ are maximal rectangular skew lattices in $\mathbf{S}$ (cf. \cite{Le89}). Thus, the functor $S \mapsto S/\DD$ is a reflection of skew lattices into ordinary lattices. Hence, to the lattice image $S/\DD$ we now call lattice reflection.

\begin{figure}
\begin{center}
$\begin{array}{lcr}
&
\begin{pspicture}(-2,-2)(2,2)
\psline[linewidth=0.5 pt,linestyle=dashed]{*-*}(1,0)(-1,0)
\psline[linewidth=0.5pt]{*-*}(-1,0)(0,1)
\psline[linewidth=0.5pt]{*-*}(1,0)(0,1)
\psline[linewidth=0.5pt]{*-*}(1,0)(0,-1)
\psline[linewidth=0.5 pt](-1,0)(0,-1)
\uput[180](-1,0){$2$}
\uput[1](1,0){$3$}
\uput[90](0,1){$1$}
\uput[270](0,-1){$0$}
\end{pspicture}
&
\\
\begin{tabular}{ l | c c c c }
  $\wedge$ & $0$ & $2$ & $3$ & $1$ \\
  \hline
   $0$ & $0$ & $0$ & $0$ & $0$ \\
  $2$ & $0$ & $2$ & $3$ & $2$ \\
  $3$ & $0$ & $2$ & $3$ & $3$  \\
  $1$ & $0$ & $2$ & $3$ & $1$
\end{tabular}
&
&
\begin{tabular}{ l | c c c c }
  $\vee$ & $0$ & $2$ & $3$ & $1$ \\
  \hline
  $0$ & $0$ & $2$ & $3$ & $1$ \\
  $2$ & $2$ & $2$ & $2$ & $1$ \\
  $3$ & $3$ & $3$ & $3$ & $1$  \\
  $1$ & $1$ & $1$ & $1$ & $1$
\end{tabular}
\end{array}$
\end{center}
\caption{The Caley tables and the admissible Hasse diagram of a right-handed skew lattice.}\label{nonstcat}
\end{figure}


\section{On the coset structure}
In the following section we shall discuss some aspects of the coset structure of a skew lattice, introduced in \cite{Le93}, and further developed in \cite{AAA80} and \cite{JPC12}.
Recall that a \emph{chain} (or \emph{totally ordered set}) is a set where each two elements are (order) related, and an \emph{antichain} is a set where no two elements are (order) related. 
We call $\mathbf S$ a \emph{skew chain} whenever $S/\DD $ is a chain. 
All $\DD $-classes are antichains for the partial order and chains for the preorder (cf. \cite{Le89}). 

Consider a skew lattice $\mathbf{S}$ consisting of exactly two $\DD$-classes $A>B$. 
Given $b\in B$, the subset $A\wedge b\wedge A=\{a\wedge b\wedge a \,|\, a\in A\}$ of $B$ is said to be a \emph{coset} of $A$ in $B$ (or an \emph{$A$-coset in $B$}). 
Similarly, a coset of $B$ in $A$ (or a $B$-coset in $A$) is any subset $B\vee a\vee B =\{b\vee a\vee b \,|\, b\in B \}$ of $A$, for a fixed $a\in A$.  
On the other hand, given $a\in A$, the \emph{image set} of $a$ in $B$ is the set $a\wedge B\wedge a = \set{a \wedge b\wedge a\,|\,b\in B}=\set{b\in B\,|\,b< a}.$ 
Dually, given $b\in B$ the set $b\vee A\vee b = \set{a\in A:b<a}$ is the image set of $b$ in $A$.  

\begin{theorem} \cite{Le93}  \label{coset_part}
Let $\mathbf S$ be a skew lattice with comparable $\DD$-classes $A>B$. 
Then, $B$ is partitioned by the cosets of $A$ in $B$ and the image set of any element $a\in A$ in $B$ is a transversal of the cosets of $A$ in $B$; dual remarks hold for any $b\in B$ and the cosets of $B$ in $A$ that determine a partition of $A$. 
Moreover, any coset $B\vee a\vee B$ of $B$ in $A$ is isomorphic to any coset $A\wedge b\wedge A$ of $A$ in $B$ under a natural bijection $\varphi $ defined implicitly for any $a\in A$ and $b\in B$ by: $x\in B\vee a\vee B$ corresponds to $y\in A\wedge b\wedge A$ if and only if $x\geq y$. 
Furthermore, the operations $\wedge$ and $\vee$ on $A\cup B$ are determined jointly by the coset bijections and the rectangular structure of each $\DD$-class.
\end{theorem}

\begin{proposition}\label{strg_prop} \cite{Co09a} 
Let $\mathbf{S}$ be a skew lattice with comparable $\DD$-classes $A>B$
and let $y,y'\in B$. The following are equivalent:
\begin{itemize}
\item[(i)] $A\wedge y \wedge A = A\wedge y' \wedge A$;
\item[(ii)] for all $x\in A$, $x\wedge y\wedge x = x\wedge y'\wedge x$;
\item[(iii)] there exists $x\in A$ such that $x\wedge y\wedge x = x\wedge y'\wedge x$.
\end{itemize}

Dual results hold, having a similar statement.
\end{proposition}

\begin{lemma}\label{rect} 
Let $\mathbf S$ be a skew lattice.
Then, the rectangularity of $\wedge$ (dually, of $\vee$) implies the rectangularity of $\mathbf S$.
Moreover, it is equivalent to the validity of the identity $x\wedge y = y\vee x$. 
\end{lemma}

\begin{proof}
Let $\mathbf S$ be a skew lattice and $x,y,z\in S$. 
Assuming the rectangularity of $\wedge$ (i.e. $x\wedge y\wedge z=x\wedge z$) we get
\[
(x\vee y)\wedge x 	= (x\vee y)\wedge y\wedge x =y\wedge x
\]
where the first equality issue to the assumption and the second to absorption, and also
\[
(x\vee y)\wedge x 	= (x\vee y)\wedge x \wedge (x\vee y) = (x\vee y)\wedge (x\vee y) = x\vee y  
\]
where again the first equality is due to absorption, the second is due to the assumption, and the third follows by idempotency.
Thus, we get that 
\[
x\vee y\vee z 	 = x\vee (z\wedge y) = z\wedge y\wedge x  = z\wedge x = x\vee z 
\]
The proof that the rectangularity of $\vee $ implies the rectangularity of the skew lattice is now similar. 
\end{proof}

Given two partially ordered sets $(S, \leq)$ and $(T, \leq)$, a function $f: S \longrightarrow T$ is an order-embedding if $f$ is both order-preserving and order-reflecting, i.e. for all $x,y \in S$, $x\leq y$ if and only if $f(x)\leq f(y)$.
As for lattices, an order isomorphism can be characterized as a surjective order-embedding. 
Any order-embedding $f$ restricts to an isomorphism between its domain $S$ and its range $f(S)$. 

\begin{lemma}\label{recthom}
Let $\mathbf S$ be a skew lattice with $\DD$-classes $A$ and $B$. 
A map $\phi:A\rightarrow B$ is a $\wedge$-homomorphism if and only if it is a $\vee$-homomorphism.
\end{lemma}

\begin{proof}
Let $x,y\in A$ and assume that $\phi$ is a $\wedge$-homomorphism. Then, 
\[
\phi(x\vee y)=\phi(y\wedge x)=\phi(y)\wedge \phi(x)=\phi(x)\vee \phi(y)
\] 
due to the rectangularity of $A$ and $B$, respectively. The converse is analogous. 
\end{proof}

\begin{proposition}\label{cosetrect} 
Both cosets and image sets form rectangular subalgebras of their relevant $\DD$-classes. Moreover, all coset bijections are isomorphisms of cosets.
\end{proposition}

\begin{proof}
Consider two comparable $\DD$-classes $A>B$ in a skew lattice $\mathbf S$.
Given $a\wedge b\wedge a',a''\wedge b\wedge a'''\in A\wedge b\wedge A$, $(a\wedge b\wedge a')\wedge (a''\wedge b\wedge a''')=a\wedge b\wedge a'''$ in $A\wedge b\wedge A$, due to the rectangularity of the $\DD$-classes.
Likewise, given $a\wedge b\wedge a, a\wedge b'\wedge a\in a\wedge B\wedge a$, $(a\wedge b\wedge a)\wedge (a\wedge b'\wedge a)=a\wedge b\wedge b'\wedge a \in a\wedge B\wedge a$.
That both $A\wedge b\wedge A$ and the image if $a$ in $B$ are also closed under $\vee$ follows from the rectangular identity, $x\wedge y=y\wedge x$ given in Lemma \ref{rect}. 
We shall now show that each $A$-coset in $B$ is isomorphic to any $B$-coset in $A$.
Let $a\in A$ and $b\in B$ and consider the bijection $\phi_{a,b}$ between $B\vee a\vee B$ and $A\wedge b\wedge A$. 
Recall that $\phi_{a,b}$ is defined for each $x\in B\vee a\vee B$ by $\phi_{a,b}(x)=x\wedge b\wedge x$. Thus, for each $x,y \in B\vee a\vee B$, we get: 

\begin{align}
\phi_{a,b}(x)\wedge \phi_{a,b}(y)	&=x\wedge b\wedge x\wedge y\wedge b\wedge y  \nonumber  \\
                                      			&=x\wedge y\wedge x\wedge b\wedge x\wedge y\wedge b\wedge y\wedge x\wedge y   \nonumber  \\
                                      			&=x\wedge y\wedge b\wedge x\wedge y\wedge b\wedge x\wedge y  \nonumber   \\
                                      			&=x\wedge y\wedge b\wedge b\wedge x\wedge y   \nonumber  \\
                                      			&=x\wedge y\wedge b\wedge x\wedge y  \nonumber   \\
                                      			&= \phi_{a,b}(x\wedge y)  \nonumber  
\end{align}
due to the regularity of $\wedge$. Hence, $\phi_{a,b}$ is an isomorphism due to Lemma \ref{recthom}.
\end{proof}

\begin{proposition}\label{newcong}
Let $B$ and $C$ be distinct $\DD$-classes in a skew lattice $\mathbf S$ such that $B>C$. 
Consider the relation $\theta_{[C:B]}$ defined in $C$ by 
\[
x\theta_{[C:B]} y\text{  iff  }x,y\in B\wedge c\wedge B\text{,  for some  } c\in C.
\] 
This relation is the equivalence corresponding to the coset partition of $B$ in $C$. 
Furthermore, it is a congruence of $B$. 
Dually, the equivalence $\theta_{[B:C]}$ derived from the coset partition of $C$ in $B$ is a congruence of $C$.
\end{proposition}

\begin{proof}
Let $x,y,z,w\in C$ such that $x\theta_{[B:C]} y$ and $z\theta_{[B:C]} w$. 
Fix $b\in B$. 
Then, Proposition \ref{strg_prop} implies that $b\wedge x\wedge b=b\wedge y\wedge b$ and that $b\wedge z\wedge b=b\wedge w\wedge b$. 
Thus, 
$\begin{array}{lcl}
b\wedge x\wedge z\wedge b &=& b\wedge x\wedge b\wedge z\wedge b  \\
                                                     &=& b\wedge x\wedge b\wedge b\wedge z\wedge b  \\
                                                     &=& b\wedge y\wedge b\wedge b\wedge w\wedge b  \\
                                                     &=& b\wedge y\wedge b\wedge w\wedge b  \\
                                                     &=& b\wedge y\wedge w\wedge b 
\end{array}$
due to the assumption and to regularity. 
On the other hand, $b\wedge (x\vee z)\wedge b=b\wedge z\wedge x\wedge b=b\wedge w\wedge y\wedge b=b\wedge (y\vee w)\wedge b$ due to the rectangularity of $B\wedge z\wedge x\wedge B$ and to the assumption. Hence, $\theta_{[B:C]}$ is a congruence of $C$.
The dual statement has a similar proof.
\end{proof}

\begin{remark}
Given any distinct $\DD$-classes $A>B$ in a skew lattice $\mathbf S$, each $A$-coset in $B$ is a subskew lattice of $\mathbf S$ isomorphic of any $B$-coset in $A$.
This is due to Propositions \ref{cosetrect} and \ref{newcong}.
Moreover, the equivalence determined in $B$ by $A$ is a congruence of $A$ with $A$-coset in $B$ as its equivalence classes.
Similar remarks hold in the dual case.
\end{remark}


\section{The coset category}
A skew lattice is \emph{categorical} if nonempty composites of coset bijections are coset bijections. 
Rectangular and normal skew lattices are examples of categorical skew lattices (cf.  \cite{Le93}). 

\begin{example}\label{ex_noncat}
A minimal example of a non categorical skew lattice is given by the 8-element left-handed skew chain given in Figure \ref{fig_skewdiag} and discussed in \cite{Le13}.
In that example, considering the skew chain $\set{0,4}>\set{3,6,1,7}>\set{2,5}$, the coset bijections are the following:

\begin{center}
$\begin{array}{c}
\varphi_{1} : \set{0,4}\rightarrow \set{3,1}\text{,  }\varphi_{2} : \set{0,4}\rightarrow \set{6,7}\\

\psi_{1}:\set{3,7}\rightarrow \set{2,5}\text{, }\psi_{2}:\set{6,1}\rightarrow \set{2,5}\\

\chi :\set{0,4}\rightarrow \set{2,5}  

\end{array}$
\end{center}

Observe that $0$ has no image by $\psi_{2}\circ \varphi_{1}$ and that $\chi(0)\in \set{2,5}$ so that $\psi_{2}\circ \varphi_{1}\neq \chi$.
The reader can find a detailed study of such examples in \cite{Le13}  where this skew lattice is named $X_{2}$ and its right-handed version is named $Y_{2}$.  
A family of non categorical examples where $X_{2}$ and $Y_{2}$ belong was further studied in \cite{Le13}.
\end{example}

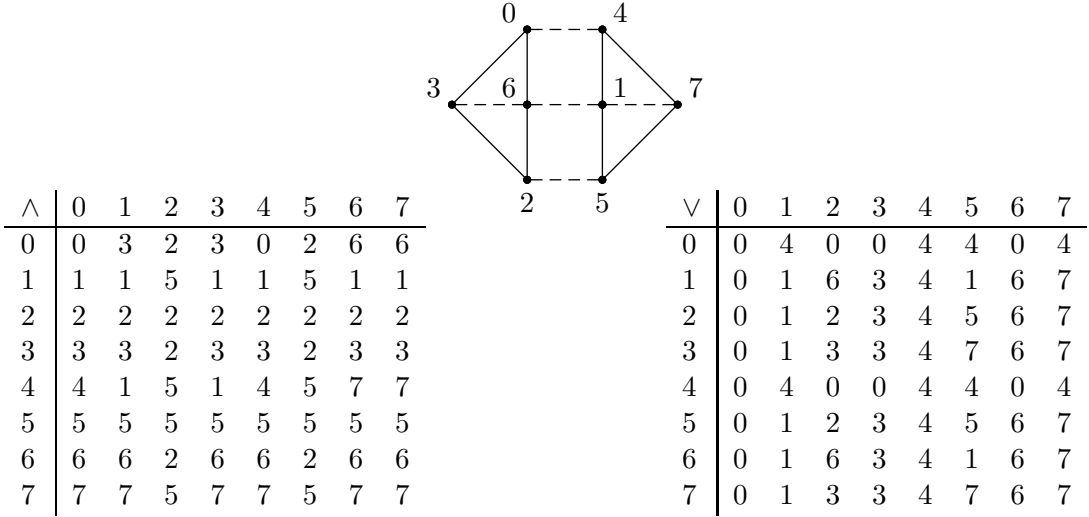
\begin{figure}  
\begin{center}  
$\begin{array}{ccc}
&
\begin{pspicture}(-1.5,-1)(1,1.5) 
\psline[linewidth=0.5 pt,linestyle=dashed]{*-*}(-1.5,0)(-0.5,0) 
\psline[linewidth=0.5 pt,linestyle=dashed]{*-*}(0.5,0)(1.5,0)
\psline[linewidth=0.5 pt,linestyle=dashed]{*-*}(-0.5,1)(0.5,1)
\psline[linewidth=0.5 pt,linestyle=dashed]{*-*}(-0.5,-1)(0.5,-1)
\psline[linewidth=0.5 pt,linestyle=dashed]{*-*}(-0.5,0)(0.5,0)
\psline[linewidth=0.5pt]{*-*}(-1.5,0)(-0.5,1) 
\psline[linewidth=0.5 pt]{*-*}(-0.5,0)(-0.5,1) 
\psline[linewidth=0.5pt]{*-*}(0.5,0)(0.5,1) 
\psline[linewidth=0.5 pt]{*-*}(1.5,0)(0.5,1)
\psline[linewidth=0.5pt]{*-*}(0.5,0)(0.5,-1)
\psline[linewidth=0.5pt]{*-*}(1.5,0)(0.5,-1) 
\psline[linewidth=0.5pt]{*-*}(-0.5,0)(-0.5,-1) 
\psline[linewidth=0.5pt]{*-*}(-1.5,0)(-0.5,-1) 
\uput[140](-0.5,0){$6$} 
\uput[140](-1.5,0){$3$} 
\uput[40](0.5,0){$1$}
\uput[40](1.5,0){$7$}  
\uput[40](0.5,1){$4$} 
\uput[140](-0.5,1){$0$} 
\uput[270](-0.5,-1){$2$} 
\uput[270](0.5,-1){$5$}
\end{pspicture}
&
\\
\begin{tabular}{ l | cccccccc }
  $\wedge$ & 0 & 1 & 2 & 3 & 4 & 5 & 6 & 7  \\
  \hline
  0                 & 0 & 3 & 2 & 3 & 0 & 2 & 6 & 6 \\
  1                 & 1 & 1 & 5 & 1 & 1 & 5 & 1 & 1 \\
  2                 & 2 & 2 & 2 & 2 & 2 & 2 & 2 & 2 \\
  3                 & 3 & 3 & 2 & 3 & 3 & 2 & 3 & 3 \\
  4                 & 4 & 1 & 5 & 1 & 4 & 5 & 7 & 7 \\  
  5                 & 5 & 5 & 5 & 5 & 5 & 5 & 5 & 5 \\
  6                 & 6 & 6 & 2 & 6 & 6 & 2 & 6 & 6 \\
  7                 & 7 & 7 & 5 & 7 & 7 & 5 & 7 & 7 
\end{tabular}
&
&
\begin{tabular}{ l | cccccccc }
  $\vee$ & 0 & 1 & 2 & 3 & 4 & 5 & 6 & 7  \\
  \hline
  0                 & 0 & 4 & 0 & 0 & 4 & 4 & 0 & 4 \\
  1                 & 0 & 1 & 6 & 3 & 4 & 1 & 6 & 7 \\
  2                 & 0 & 1 & 2 & 3 & 4 & 5 & 6 & 7 \\
  3                 & 0 & 1 & 3 & 3 & 4 & 7 & 6 & 7 \\
  4                 & 0 & 4 & 0 & 0 & 4 & 4 & 0 & 4 \\  
  5                 & 0 & 1 & 2 & 3 & 4 & 5 & 6 & 7 \\
  6                 & 0 & 1 & 6 & 3 & 4 & 1 & 6 & 7 \\
  7                 & 0 & 1 & 3 & 3 & 4 & 7 & 6 & 7 
\end{tabular}
\end{array}$
\caption{The admissible Hasse diagram of a left-handed non categorical skew lattice.} 
\label{fig_skewdiag} 
\end{center}  
\end{figure}

A categorical skew lattice is \emph{strictly categorical} if the compositions of coset bijections between comparable $\DD $-classes $A>B>C$ are never empty. 
Rectangular and normal skew lattices are strictly categorical skew lattices (cf. \cite{Le93}). 
In particular, subskew lattices of strictly categorical skew lattices are also strictly categorical.  

\begin{proposition}\cite{Ka05}\label{prop_bigid}   
A skew chain $\mathbf S$ consisting of $\DD $-classes $A>B>C$ is categorical iff for all elements $a\in A$, $b\in B$ and $c\in C$ satisfying $a>b>c$, one (and hence both) of the following equivalent statements holds:

\begin{itemize}
\item[(i)]  $(A\wedge b\wedge A)\cap (C\vee b\vee C)=(C\vee a\vee C)\wedge b\wedge (C\vee a\vee C)$; 
\item[(ii)]  $(A\wedge b\wedge A)\cap (C\vee b\vee C)=(A\wedge c\wedge A)\vee b\vee (A\wedge c\wedge A).$
\end{itemize}

Moreover, $\mathbf S$ is strictly categorical iff in addition to
\textup{(i)--(ii)}, for all $b,b'\in B$,
\[
(A\wedge b\wedge A)\cap (C\vee b'\vee C)\neq \emptyset .
\]

\end{proposition}

The following is a practical criteria to identify strictly categorical skew lattices.

\begin{proposition}\cite{Le11a}\label{strictbug} 
A skew chain $A > B > C$ is strictly categorical if and only if given $a\in A$, $b,b'\in B$ and $c\in C$ such that $a>b>c$ and $a>b'>c$, then $b=b'$ must follow.
\end{proposition}

\begin{example}\label{exminimal}
A minimal example of a categorical skew lattice that is not strictly categorical is given by the right-handed manifestation of the skew chain with three $\DD $-classes in Figure \ref{nonstcat}. 
In fact, the composition of the coset bijections $\psi :\set{1}\rightarrow \set{2}$ and $\varphi ':\set{3}\rightarrow \set{0}$ is empty.
Observe that $\chi:\set{1}\rightarrow \set{0}$ can be decomposed either by the composition of $\psi $ and $\varphi:\set{2}\rightarrow \set{0}$, or by the composition of $\psi ' :\set{1}\rightarrow \set{3}$ and $\varphi '$ (cf. \cite{Le11a}).
\end{example}

\begin{proposition}\cite{Co11}\label{cs_normal} 
Let $\mathbf S$ be a skew lattice. 
Then, $\mathbf S$ is normal iff for each comparable pair of $\DD $-classes $A>B$ in $\mathbf S$, $B$ is the entire coset of $A$ in $B$. 
That is, for all $x,x'\in B$, 
\[
A\wedge x\wedge A=A\wedge x'\wedge A.
\]
Dually, $\mathbf S$ is conormal iff for all comparable pairs of $\DD $-classes $A>B$ in $\mathbf S$ and all $x,x'\in A$, $B\vee x\vee B=B\vee x'\vee B$.
\end{proposition}

\begin{example}\label{nonorm}
Strictly categorical skew lattices need not be normal: any skew lattice with the admissible Hasse diagram below represents a right-handed skew chain and thus a strictly categorical skew lattice (consider, for instance, the subskew lattice $\set{1,2,3}$ of the skew lattice in Example \ref{nonstcat}.
\begin{center}
\begin{pspicture}(-1,-0.5)(1,1.5)
\psline[linewidth=0.5 pt,linestyle=dashed]{*-*}(1,0)(-1,0)
\psline[linewidth=0.5pt]{*-*}(-1,0)(0,1)
\psline[linewidth=0.5pt]{*-*}(1,0)(0,1)
\uput[180](-1,0){$2$}
\uput[1](1,0){$3$}
\uput[90](0,1){1}
\end{pspicture}
\end{center}
Normality fails as the upper $\DD $-class determines more then one coset in the lower $\DD $-class: observe that, considering $A=\set{1}$ and $B=\set{2,3}$, $A>B$ is a strictly categorical skew chain, according to Proposition \ref{strictbug}, but 
\[
A\wedge 2\wedge A=\set{2}\neq \set{3}=A\wedge 3\wedge A.
\]
\end{example}


The prefix \textit{categorical} was motivated by the definition of categorical skew lattices as the ones for whom coset bijections form a category under certain conditions: being strictly categorical.
This category was first introduced in \cite{Le93}, mentioned as \emph{category of coset bijections}, defined as follows: 

\begin{definition}\label{cosetcat}
Let $\mathbf S$ be a strictly categorical skew lattice.
Define the \emph{coset category}, denoted by $\CC $, as given by:
\begin{itemize}
\item[] the class of objects of $\CC$ is the set of all the $\DD $-classes of $\mathbf S$ (endowed with their rectangular structure);
\item[] for strictly comparable $\DD$-classes $A > B$, $\CC (A, B)$ is the set of all the coset bijections from the $B$-cosets in $A$ to the $A$-cosets in $B$ (that are isomorphisms). Otherwise, $\CC (A, B)$ consists of the empty bijection; 
\item[] $\CC (A, A)$ consists of the unique identity bijection on $A$;
\item[] morphism composition is the usual composition of partial bijections. 
\end{itemize}
\end{definition}

The category is modified in case $\mathbf S$ is categorical but not strictly categorical by adding the requirement that, for each pair $A\geq B$, $\CC(A,B)$ contains the empty bijection. 
And in the case of empty composites an $A-B$ labelled copy of the empty partial bijection with empty composites, given the appropriate labeling to avoid confusing empty partial bijections in different morphism sets.

\begin{remark}\cite{AAA80}
All $\DD $-classes in a skew lattice form antichains. Thus, in a categorical skew lattice  $\mathbf S$, for all comparable $\DD $-classes $A\geq B$,
\[
\cup \CC (A,B) = \cup \set{\phi_{a,b}:B\vee a\vee B\rightarrow A\wedge b\wedge A \text{ coset bijection }\mid a\in A,b\in B} =\geq _{A\times B}.
\]  

By the nature of the coset bijections, $\CC $ is a self dual category.
\end{remark}
 
The following research was proposed to us by Jonathan Leech for the purpose of the author's Ph.D. dissertation in \cite{JPC12}, and shows that not all skew lattices can determine such a category as strictly categorical skew lattices do.

Consider a skew chain $\set{A>B>C}$, $x\in A$, $y,y'\in B$, $z\in C$ and $\varphi:B\vee x\vee B\rightarrow A\wedge y\wedge A$, $\psi:C\vee y'\vee C\rightarrow B\wedge z\wedge B$ coset bijections. 
As neither $\psi$ nor $\varphi$ are empty, if $\psi   \varphi$ is empty then $\text{dom}(\psi)\cap im(\varphi)=(C\vee y'\vee C)\cap (A\wedge y\wedge A)=\emptyset$.
Otherwise, $\psi   \varphi$ is nonempty so take $a\in B\vee x\vee B$, $b\in (C\vee y'\vee C)\cap (A\wedge y\wedge A)$ and $c\in B\wedge z\wedge B$ such that $a>b>c$ and that $\varphi =\phi_{a,b}$, $\psi =\phi_{b,c}$ and $\chi = \phi_{a,c}$. 
Observe that
$\begin{array}{lcl}
\text{dom}(\phi_{a,b}   \phi_{b,c}^{-1}) &= &\phi_{a,b}^{-1} ((A\wedge b\wedge A)\cap (C\vee b\vee C))	\\
                                               & \subseteq &\phi_{a,b}^{-1} (C\vee b\vee C)	\\
                                              & = & (C\vee b\vee C)\vee a\vee (C\vee b\vee C)		\\
                                               &\subseteq & C\vee a\vee C  
\end{array}$
and that, for all $e\in B\vee a\vee B$, 
$\begin{array}{lcl}
\phi_{b,c}   \phi_{a,b} & = & \phi_{b,c} (\phi_{a,b} (e) = \phi_{b,c} (e\wedge b\wedge e)	 \\
                            & = &(e\wedge b\wedge e)\wedge c\wedge (e\wedge b\wedge e) = e\wedge b\wedge c\wedge b\wedge e \\
                            & = & e\wedge c\wedge e	=\phi_{a,c} (e)
\end{array}$
Hence, there is a unique coset bijection $\chi:A\rightarrow C$ containing $\psi   \varphi\neq \emptyset$. We denote this coset bijection by $\psi \times \varphi$. Whenever, $\psi   \varphi= \emptyset$ we say that $\psi \times \varphi = \emptyset$.

In fact, $\psi \times \varphi$ needs not be the direct composite of partial functions if $\mathbf S$ is not categorical. 
Moreover, $\chi:A\rightarrow C$ always exists as a coset bijection and always contains $\psi \times \varphi$, which is trivial in the case that $\psi \times \varphi$ is empty.

Given a skew lattice $\mathbf S$, define a new pre-category $\EE$ for which the objects are the $\DD$-classes of $\mathbf S$ and the morphisms between two objects (i.e., $\DD$-classes) $A$ and $B$ is the set $\EE(A, B)$ of all coset bijections from $A$ to $B$. $\EE (A, A)$ is the unique identity bijection on an object $A$ and the composition of two morphisms $\EE(A,B)$ and $\EE(B,C)$ is the morphism that includes compositions of all coset bijections from $A$ to $B$ with coset bijections from $B$ to $C$. In other words, for all objects $A$, $B$ and $C$, 
\[
\EE(A,B)   \EE(B,C) =\set{ \psi \times \varphi \mid \varphi:A\rightarrow B\text{ and }\psi:B\rightarrow C\text{ are coset bijections }}.
\]  

Observe that in the skew lattice represented in Figure \ref{leech} we have the following cosets:
\begin{center}
\begin{align} 
B\vee 0\vee B=&\set{0,4}=B\vee 4\vee B  & &  A\wedge 2\wedge A=\set{2,5}=A\wedge 5\wedge A \nonumber \\
A\wedge 3\wedge A=&\set{1,3}=A\wedge 1\wedge A  & & D\vee 3\vee D=\set{3}   \nonumber \\
A\wedge 6\wedge A=&\set{6,7}=A\wedge 7\wedge A  & & D\vee 6\vee D=\set{6}   \nonumber \\
C\vee 3\vee C=&\set{3,7}=C\vee 7\vee C    & &  D\vee 1\vee D=\set{1}  \nonumber \\
C\vee 6\vee C=&\set{1,6}=C\vee 1\vee C      & & D\vee 7\vee D=\set{7}   \nonumber \\
B\wedge 2\wedge B=&\set{2,5}=B\wedge 5\wedge B	& & B\wedge 8\wedge B=\set{8}   \nonumber \\
D\vee 2\vee D=&\set{2}  & & D\vee 0\vee D=\set{0}  \nonumber  \\
D\vee 5\vee D=&\set{5}  & & D\vee 4\vee D=\set{4}  \nonumber  \\
C\wedge 8\wedge C=&\set{8}  & &  A\wedge 8\wedge A=\set{8} \nonumber  \\
C\vee 0\vee C=&\set{0,4}=C\vee 4\vee C  & &    \nonumber
\end{align}
\end{center}
We, therefore, are able to consider the corresponding coset bijections as follows:
\[ 
\begin{array}{lcr}
\begin{array}{lc}
\phi_{0,6}:&0\rightarrow 6\\
                  &4\rightarrow 7
\end{array} 
&&
\begin{array}{lc}
\phi_{3,2}:&7\rightarrow 5\\
                  &3\rightarrow 2
\end{array} 
\\
\begin{array}{lc}
\phi_{2,0}:&2\rightarrow 8
\end{array}        
&&
\begin{array}{lc}
\phi_{3,2}\times \phi_{0,6}:&0\rightarrow 2\\
                                               &4\rightarrow 5
\end{array} 
\\
\begin{array}{lc}
\phi_{2,0}\times (\phi_{3,2}\times \phi_{0,6}):&0\rightarrow 8
\end{array} 
&&
\begin{array}{lc}
\phi_{2,0}\times \phi_{3,2}:&3\rightarrow 0
\end{array}  
\end{array} 
\]
In this case, $\text{dom} (\phi_{2,0}\times \phi_{3,2})=\set{3}$ and $3\notin \set{6,7}=im( \phi_{0,6})$. 
Thus, $(\phi_{2,0}\times \phi_{3,2})    \phi_{0,6} =\emptyset $ so that  $(\phi_{2,0}\times \phi_{3,2})\times   \phi_{0,6} =\emptyset $.
As $(0,8)\in \phi_{2,8}\times (\phi_{3,2}\times \phi_{0,6})$, we conclude that 
\[
\phi_{2,8}\times (\phi_{3,2}\times \phi_{0,6})\neq (\phi_{2,8}\times \phi_{3,2})\times \phi_{0,6}.
\]

Hence, $\delta \times \psi \times \varphi$ is not uniquely determined as $\times$ is not associative and, therefore, $\EE$ can not constitute a category.
In fact, the mentioned skew lattice is the result of taking the non categorical skew lattice $X_{2}$ and adjoin a $0$ element at the bottom. 
Analogously we could also had accomplished such an example by adjoining a $1$ element at the top instead. 
These four examples are likely to be the minimal possible examples.

\begin{figure}  
\begin{center}  
$\begin{array}{lll}
&
\scalebox{0.8}{
\begin{pspicture}(-1.5,-2)(1.5,2) 
\psline[linewidth=0.5 pt,linestyle=dashed]{*-*}(-1.5,0)(-0.5,0) 
\psline[linewidth=0.5 pt,linestyle=dashed]{*-*}(0.5,0)(1.5,0)
\psline[linewidth=0.5 pt,linestyle=dashed]{*-*}(-0.5,1)(0.5,1)
\psline[linewidth=0.5 pt,linestyle=dashed]{*-*}(-0.5,-1)(0.5,-1)
\psline[linewidth=0.5 pt,linestyle=dashed]{*-*}(-0.5,0)(0.5,0)
\psline[linewidth=0.5pt]{*-*}(-1.5,0)(-0.5,1) 
\psline[linewidth=0.5 pt]{*-*}(-0.5,0)(-0.5,1) 
\psline[linewidth=0.5pt]{*-*}(0.5,0)(0.5,1) 
\psline[linewidth=0.5 pt]{*-*}(1.5,0)(0.5,1)
\psline[linewidth=0.5pt]{*-*}(0.5,0)(0.5,-1)
\psline[linewidth=0.5pt]{*-*}(1.5,0)(0.5,-1) 
\psline[linewidth=0.5pt]{*-*}(-0.5,0)(-0.5,-1) 
\psline[linewidth=0.5pt]{*-*}(-1.5,0)(-0.5,-1) 
\psline[linewidth=0.5 pt]{*-*}(-0.5,-1)(0,-1.5)
\psline[linewidth=0.5 pt]{*-*}(0.5,-1)(0,-1.5)
\uput[140](-0.5,0){$6$} 
\uput[140](-1.5,0){$3$} 
\uput[40](0.5,0){$1$}
\uput[40](1.5,0){$7$}  
\uput[40](0.5,1){$4$} 
\uput[140](-0.5,1){$0$} 
\uput[270](-0.5,-1){$2$} 
\uput[270](0.5,-1){$5$}
\uput[270](0,-1.5){$8$}
\end{pspicture}
}
&
\\
\scalebox{0.8}{
\begin{tabular}{ l | ccccccccc }
  $\wedge$ & 0 & 1 & 2 & 3 & 4 & 5 & 6 & 7 & 8 \\
  \hline
  0                 & 0 & 3 & 2 & 3 & 0 & 2 & 6 & 6 &8\\
  1                 & 1 & 1 & 5 & 1 & 1 & 5 & 1 & 1 &8\\
  2                 & 2 & 2 & 2 & 2 & 2 & 2 & 2 & 2 &8\\
  3                 & 3 & 3 & 2 & 3 & 3 & 2 & 3 & 3 &8\\
  4                 & 4 & 1 & 5 & 1 & 4 & 5 & 7 & 7 &8\\  
  5                 & 5 & 5 & 5 & 5 & 5 & 5 & 5 & 5&8\\
  6                 & 6 & 6 & 2 & 6 & 6 & 2 & 6 & 6 &8\\
  7                 & 7 & 7 & 5 & 7 & 7 & 5 & 7 & 7 &8\\
  8                 & 8 & 8 & 8 & 8 & 8 & 8 & 8 & 8 &8\\
\end{tabular}
}
&
&
\scalebox{0.8}{
\begin{tabular}{ l | ccccccccc }
  $\vee$ & 0 & 1 & 2 & 3 & 4 & 5 & 6 & 7& 8  \\
  \hline
  0                 & 0 & 4 & 0 & 0 & 4 & 4 & 0 & 4 & 0 \\
  1                 & 0 & 1 & 6 & 3 & 4 & 1 & 6 & 7 &1 \\
  2                 & 0 & 1 & 2 & 3 & 4 & 5 & 6 & 7 &2 \\
  3                 & 0 & 1 & 3 & 3 & 4 & 7 & 6 & 7 &3 \\
  4                 & 0 & 4 & 0 & 0 & 4 & 4 & 0 & 4 &4 \\  
  5                 & 0 & 1 & 2 & 3 & 4 & 5 & 6 & 7 &5 \\
  6                 & 0 & 1 & 6 & 3 & 4 & 1 & 6 & 7 &6 \\
  7                 & 0 & 1 & 3 & 3 & 4 & 7 & 6 & 7 &7 \\
  8                 & 0 & 1 & 2 & 3 & 4 & 5 & 6 & 7 & 8
\end{tabular}
}
\end{array}$
\caption{\small \sl The admissible Hasse diagram of another left-handed non categorical skew lattice for which $\phi_{2,8}\times (\phi_{3,2}\times \phi_{0,6})\neq (\phi_{2,8}\times \phi_{3,2})\times \phi_{0,6}$.} \label{leech} 
\end{center}  
\end{figure}


\section{Strictly categorical skew lattices in rings}
Let $\mathbf{R}=(R,+,\cdot)$ be a ring and $E(\mathbf{R})$ the set of all idempotent elements in $\mathbf{R}$. 
Set $x\wedge y=xy$ and $x\vee y=x\circ y=x+y-xy$. $x\circ y$ needs not be idempotent (cf. \cite{Ka05}).
A regular band $B$ in a ring does not generate a skew lattice in general but the assertion is true if $B$ satisfies a stronger identity, that is if $B$ is in fact a normal band.
If $S\subseteq E(\mathbf{R})$ is closed under both $\cdot $ and $\circ $ then $(S;\cdot , \circ )$ is a skew lattice.  
Another possible choice for the operation $\vee$ is  $\nabla$, defined by, 
\[
x\nabla y=(x\circ y)^2=x+y+yx-xyx-yxy,
\]
In general, the operation $\nabla$ needs not be associative (conf. \cite{Ka05c}, ex 2.1).
Though, $\nabla $ is associative in the presence of normality (conf. \cite{Ka05c}, prop 2.2). 
By a \emph{skew lattice in a ring}\index{skew lattice in a ring} $\mathbf R$ we mean a set $S\subseteq E(\mathbf R)$ that is closed under both multiplication and $\nabla$, and forms a skew lattice for the two operations. 
In particular, we have to make sure that $\nabla$ is associative in $\mathbf S$. 
Given a multiplicative band $\mathbf B$ in a ring $\mathbf R$ the relation between $\circ$ and $\nabla$ is given by $e\nabla f=(e\circ f)^2$ for all $e,f\in B$. 
Thus, $a\circ b$ and $a \nabla b$ coincide whenever $a\circ b$ is idempotent. 
In the case of right-handed skew lattices the nabla operation reduces to the circle operation.  
Any normal multiplicative band of idempotents in a ring generates a skew lattice under multiplication and the operation $\nabla$ with the reduct $(S, \cdot)$ also being normal. 
The converse is however false, that is, a skew lattice whose multiplicative reduct is not normal exist (cf. \cite{Ka05c}). 
Hence, skew lattices in rings need not be normal.

The standard form for pure bands in matrix rings was developed by Fillmore at al. in \cite{Fi94} and \cite{Fi99}. 
Based on it, Cvetko-Vah described in \cite{Ka07} the standard form for right-handed skew lattices in $M_n(F)$ as follows: 
let $E_1< \dots <E_m$ be a maximal chain of $\DD$-classes of the skew lattice $\mathbf S$. 
Then a basis for $F^n$ exists such that in this basis, for any three matrices $a\in E_i$, $b\in E_j$ and $c\in E_k$, $i>j>k$, a block decomposition exists such that $a$, $b$ and $c$ have block forms
\begin{center}
 $\begin{array}{cccc}
a= \begin{bmatrix}
I & 0 & 0 & a_{14} \\
0 & I & 0 & a_{24} \\
0 & 0 & I & a_{34} \\
0 & 0 & 0 & 0
\end{bmatrix},
&
b= \begin{bmatrix}
I & 0 & b_{13} & b_{14} \\
0 & I & b_{23} & b_{24}\\
0 & 0 & 0 & 0 \\
0 & 0 & 0 & 0
\end{bmatrix}&
\text{ and  } &
c= \begin{bmatrix}
I & c_{12} & c_{13} & c_{14} \\
0 & 0 & 0 & 0 \\
0 & 0 & 0 & 0 \\
0 & 0 & 0 & 0
\end{bmatrix}.
\end{array}$
\end{center}

\begin{lemma}\label{ab}
Given a skew chain $A>B>C$ in $M_n(F)$ consider matrices $a\in A$, $b\in B$ and $c\in C$ in the above block form
such that $a>b>c$. 
Then,  
\begin{align}
a_{14}+b_{13}a_{34}=b_{14}	& &	b_{13}+c_{12}b_{23}=c_{13},	\nonumber  \\
a_{24}+b_{23}a_{34}=b_{24}	& &	b_{14}+c_{12}b_{24}=c_{14}	\nonumber 
\end{align}
\end{lemma}

\begin{proof}
Given matrices $a$, $b$ and $c$, 
\[
ba= \begin{bmatrix}
I & 0 & b_{13} & a_{14} + b_{13}a_{34} \\
0 & I & b_{23} & a_{24}+b_{23}a_{34} \\
0 & 0 & 0 & 0 \\
0 & 0 & 0 & 0
\end{bmatrix},
\]
\[
cb= \begin{bmatrix}
I & c_{12} & b_{13}+c_{12}b_{23} & b_{14}+c_{12}b_{24} \\
0 & 0 & 0 & 0\\
0 & 0 & 0 & 0 \\
0 & 0 & 0 & 0
\end{bmatrix}.
\]
Now observe that $ba=b$ and $cb=c$ imply that the equations above hold.
\end{proof}

\begin{lemma}\label{ac}
A skew chain $A>B>C$ is strictly categorical if and only if, for all $b\in B$, there exists $b'\in B$ such that 
\begin{itemize}
\item[(i)] $b_{13}+c_{12}b'_{23}=c_{13}$
\item[(ii)] $a_{24}+b_{23}a_{34}=b'_{24}$
\item[(iii)] $a_{14}+b_{13}a_{34}=c_{14}-c_{12}b'_{24}$
\end{itemize}
where $a_{ij}\in A$, $b_{ij},b'_{ij}\in B$ and $c_{ij}\in C$.
\end{lemma}

\begin{proof}
Consider $A=E_i$, $B=E_j$ and $C=E_l$ and fix $a_{k}\in A$, $b_{k}\in B$ and $c_{k}\in C$ such that $a_{k}>b_{k}>c_{k}$ presented below:

\begin{center}
$\begin{array}{cccc}
a_{k}= \begin{bmatrix}
I & 0 & 0 & p\\
0 & I & 0 & q \\
0 & 0 & I & r \\
0 & 0 & 0 & 0
\end{bmatrix},
&
b_{k}= \begin{bmatrix}
I & 0 & w & u \\
0 & I & t & v\\
0 & 0 & 0 & 0 \\
0 & 0 & 0 & 0
\end{bmatrix}&
\text{ and  } &
c_{k}= \begin{bmatrix}
I & x & y & z \\
0 & 0 & 0 & 0 \\
0 & 0 & 0 & 0 \\
0 & 0 & 0 & 0
\end{bmatrix}.
\end{array}$
\end{center}

Given matrices $a$, $b$ and $c$, in the above block decomposition,
\[
ba= \begin{bmatrix}
I & 0 & b_{13} & a_{14} + b_{13}a_{34} \\
0 & I & b_{23} & a_{24}+b_{23}a_{34} \\
0 & 0 & 0 & 0 \\
0 & 0 & 0 & 0
\end{bmatrix},
\]
\[
c\circ b=c+b-cb=
 \begin{bmatrix}
I & 0 & c_{13}-c_{12}b_{23} & c_{14}-c_{12}b_{24} \\
0 & I & b_{23} & b_{24} \\
0 & 0 & 0 & 0 \\
0 & 0 & 0 & 0
\end{bmatrix}.
\]
so that the $A$-coset in $B$ and the $C$-coset in $B$ are given by
\begin{center}
$\begin{array}{ccc}
bA= \begin{bmatrix}
I & 0 & w & a_{14} + wa_{34} \\
0 & I & t   & a_{24}+  ta_{34} \\
0 & 0 & 0 & 0 \\
0 & 0 & 0 & 0
\end{bmatrix}
&
\text{ and  } 
&
C\circ b=
 \begin{bmatrix}
I & 0 & y-xb_{23} & z-xb_{24} \\
0 & I & b_{23} & b_{24} \\
0 & 0 & 0 & 0 \\
0 & 0 & 0 & 0
\end{bmatrix}.
\end{array}$
\end{center}
Recall that all skew lattices in rings are categorical.
Then, due to Proposition \ref{prop_bigid}, for all $b,b'\in B$, $(bA)\cap (C\circ b')\neq \emptyset$.
Thus, the above equations hold.
\end{proof}

\begin{remark}\label{catmat}
The conditions of Lemma \ref{ac} define a strictly categorical skew lattices of matrices in rings that can determine a category in the sense of Definition \ref{cosetcat}. 
A morphism in this category is determined by related $\DD$-classes $A>B$ and thus given by the set of corresponding coset bijections given in \cite{Ka07} for all matrices $a\in A$ and $b\in B$ of the block form 
\begin{center}
$\begin{array}{ccc}
a= \begin{bmatrix}
I & 0 & a_{13} \\
0 & I & a_{23} \\
0 & 0 &  0 
\end{bmatrix}
&
\text{ and  } 
&
b= \begin{bmatrix}
I & b_{12} & b_{23} \\
0 & 0  & 0 \\
0 & 0 & 0 
\end{bmatrix}.
\end{array}$
\end{center}
as the following maps:
\[
\begin{array}{ccccc}
&
\begin{bmatrix}
I & 0 & a_{13} \\
0 & I & a_{i} \\
0 & 0 &  0 
\end{bmatrix}
\rightarrow
\begin{bmatrix}
I & b_j & a_{13}+b_ja_i \\
0 & I & 0 \\
0 & 0 &  0 
\end{bmatrix}
&
\text{ and  } 
&
\begin{bmatrix}
I & 0 & 0\\
0 & I & 0 \\
a_{31} & a_i & 0 
\end{bmatrix}
\rightarrow
\begin{bmatrix}
I & 0 & 0 \\
b_j & 0 & 0 \\
a_{31}+a_ib_j & 0 &  0 
\end{bmatrix}
\end{array}
\]
\end{remark}

\begin{example}\label{nonstcatmat}
Let us now see that not all skew lattices in rings are strictly categorical.
Observe that the conditions of Lemmas \ref{ab} and \ref{ac} are very much related.
On the other hand, if $a>b,b'>c$ then $b=b'$ or else they belong to distinct cosets.
To show a counter example we will present a skew lattice with the admissible Hasse diagram of the skew lattice in the example of Figure \ref{nonstcat} with $1=a$, $0=c$, $b=2$ and $b'=3$.
To construct an example corresponding to the above diagram we have to make sure that $bb'=b'$, $b'b=b$ (which are guaranteed by assumption on their block form), $ba=b$, $b'a=b'$, $cb=c$ and $cb'=c$ implying that 
\[
\begin{tabular}{lcr}
$a_{14}+b_{13}a_{34}=b_{14}$  & & $b_{13}+c_{12}b_{23}=c_{13}$ \\
$a_{24}+b_{23}a_{34}=b_{24}$  & & $b_{14}+c_{12}b_{24}=c_{14}$ \\
$a_{14}+b'_{13}a_{34}=b'_{14}$  & & $b'_{13}+c_{12}b'_{23}=c_{13}$ \\
$a_{24}+b'_{23}a_{34}=b'_{24}$  & & $b'_{14}+c_{12}b'_{24}=c_{14}$ 
\end{tabular}
\]
Let us now consider the example where all these values are null with the exception of $b'_{23}=1$, that is,
\begin{center}
$\begin{array}{ccccc}
a= \begin{bmatrix}
1 & 0 & 0 & 0 \\
0 & 1 & 0 & 0 \\
0 & 0 & 1 & 0 \\
0 & 0 & 0 & 0
\end{bmatrix},
&
b= \begin{bmatrix}
1 & 0 & 0 & 0 \\
0 & 1 & 0 & 0\\
0 & 0 & 0 & 0 \\
0 & 0 & 0 & 0
\end{bmatrix}&
b'= \begin{bmatrix}
1 & 0 & 0 & 0 \\
0 & 1 & 1 & 0 \\
0 & 0 & 0 & 0 \\
0 & 0 & 0 & 0
\end{bmatrix}&
\text{ and  } &
c= \begin{bmatrix}
1 & 0 & 0 & 0 \\
0 & 0 & 0 & 0 \\
0 & 0 & 0 & 0 \\
0 & 0 & 0 & 0
\end{bmatrix}.
\end{array}$
\end{center}
\end{example}

\begin{example}
An example of a skew lattice of matrices in a ring $\mathbf R$ corresponding to the non normal but strictly categorical skew chain of Example \ref{nonorm} above can be given by the subskew lattice $\set{1,b,b'}$ of the skew lattice of Example \ref{nonstcatmat}.
\end{example}

\begin{proposition}\label{nrmat}
Let $\mathbf S$ be a skew lattice $A>B$ and let $a,u\in A$ and $b,v\in B$ such that

\begin{center}
$\begin{array}{ccccc}
a= \begin{bmatrix}
I & 0 & a_{13} \\
0 & I & a_{23} \\
0 & 0 &  0 
\end{bmatrix},
&
v= \begin{bmatrix}
I & x & y  \\
0 & 0 & 0\\
0 & 0 & 0 
\end{bmatrix}
&
b= \begin{bmatrix}
I & b_{12} & b_{23} \\
0 & 0  & 0 \\
0 & 0 & 0 
\end{bmatrix}
&
\text{ and  } 
&
u= \begin{bmatrix}
I & 0 & z  \\
0 & I & w  \\
0 & 0 & 0 
\end{bmatrix}.
\end{array}$
\end{center}
Then, 
\begin{itemize}
\item[(i)] $\mathbf S$ is normal iff for all $v'\Drel v$, $x=x'$ 
\item[(ii)] $\mathbf S$ is conormal iff for all $u'\Drel u$, $w=w'$  
\end{itemize}
\end{proposition}

\begin{proof}

Just observe that

\begin{center}
$\begin{array}{cc}
vA= \begin{bmatrix}
I & x & a_{13}+xa_{23} \\
0 & 0 & 0 \\
0 & 0 &  0 
\end{bmatrix},
&
B\circ u= \begin{bmatrix}
I & 0 & b_{13}-b_{12}w  \\
0 & I & w\\
0 & 0 & 0 
\end{bmatrix}
\end{array}$
\end{center}
\end{proof}

\begin{remark}
Skew lattices in rings are distributive, symmetric and categorical (cf. \cite{Le89} and \cite{Le93}). 
It is well known that $\wedge$-distributive skew lattices are exactly the skew lattices that are simultaneously symmetric and normal for which the lattice image $S/\DD$ is distributive (cf. \cite{Le92}). 
Hence, whenever $S/\DD$ is distributive, the conditions of Proposition \ref{nrmat} determine the skew lattices of matrices in rings that are $\wedge$-distributive.
\end{remark}


\section*{Acknowledgements}

We would like to thank to J. Leech that, on the later preparation of the author's Ph.D. thesis, posed the central question to this note, to K. Cvetko-Vah that carefully followed the research done to achieve the counter-example exhibited, to Andrej Bauer for the discussions subsequent to the presentation of this topic at his research seminar in Ljubljana, to Maria Jo\~ ao Gouveia for the challenging comments also during her seminar in Lisbon, and to Dijana Cerovski for all her support.  


\end{document}